\documentclass[12pt]{amsart}
\usepackage[utf8]{inputenc}

\usepackage{tikz-cd}
\usepackage{verbatim} % 
\usepackage{graphicx}
\usepackage{mathtools}
\usepackage[all]{xy} 
\usepackage{mathrsfs} 
\usepackage{enumerate}
\usepackage[T2A]{fontenc}
\usepackage[utf8]{inputenc}
\usepackage{MnSymbol}

\usepackage{cleveref}

\DeclareMathAlphabet\mathbb{U}{msb}{m}{n}

\usepackage[OT2,T1]{fontenc} 
\DeclareSymbolFont{cyrletters}{OT2}{wncyr}{m}{n}

\def\F2{{\mathbb{F}_2}}

\numberwithin{equation}{section}

\newtheorem{Theorem}{Theorem}

\newtheorem{Lemma}[Theorem]{Lemma}

\theoremstyle{definition}

\begin{document}

\title{Subquandles of free quandles}
\author{Sergei O. Ivanov}

\address{
Laboratory of Modern Algebra and Applications,  St. Petersburg State University, 14th Line, 29b,
Saint Petersburg, 199178 Russia}
\email{ivanov.s.o.1986@gmail.com}

\author{Georgii Kadantsev}
\address{Laboratory of Continuous Mathematical Education (School 564 of St. Petersburg), nab. Obvodnogo kanala
143, Saint Petersburg, Russia}
\email{kadantsev.georg@yandex.ru}

\author{Kirill Kuznetsov}
\address{Laboratory of Continuous Mathematical Education (School 564 of St. Petersburg), nab. Obvodnogo kanala
143, Saint Petersburg, Russia}
\email{kirill\_kz2001@mail.ru}

\thanks{The project is supported by the grant of the Government of the Russian Federation
for the state support of scientific research carried out under the supervision of leading scientists,
agreement 14.W03.31.0030 dated 15.02.2018}

\begin{abstract} We prove that a subquandle of a free quandle is free.
\end{abstract}
\maketitle

\section*{\bf Introduction}

A quandle is a set with two binary operations   $(Q,\triangleleft,\triangleright)$, satisfying following identities.
\begin{enumerate}
\item 
$a\triangleleft (b\triangleleft c)=(a\triangleleft b) \triangleleft (a \triangleleft c) $

\item $(a\triangleright b)\triangleright c=(a\triangleright c)\triangleright (a\triangleright c)$
    
\item $a \triangleleft (b \triangleright a)=b=(a \triangleleft b) \triangleright a$
 
\item $a\triangleleft a=a=a\triangleright a$
\end{enumerate}
A group $G$ gives an example of a quandle if we define the operations as follows:
\[ h\triangleleft g =  g^{h^{-1}}, \hspace{1cm}
g \triangleright h = g^h,
\]
where $a^b=b^{-1}ab.$
This quandle is denoted by ${\sf Conj}(G).$
The definition is motivated by knot theory (see \cite{Joyce} for details). As for any type of algebraic structure, in quandle theory there is a notion of a free quandle.  V.~Bardakov, M.~Singh and M.~Singh in 
\cite[Problem 6.12]{Bardakov} raised the question about an analogue of Nielsen–Schreier theorem for quandles: is it true that any subquandle of a free quandle is free. This note is devoted to an affirmative answer on this question: 

\

\noindent {\bf Theorem}.
{\it A subquandle of a free quandle is free.}

\

Moreover, for a subquandle $Q$ of a free quandle, we give an explicit construction of a basis  $S(Q)$ in $Q.$ The main tool for us is the following description of free quandles.

\

\noindent {\bf Theorem}~{\cite[Th.4.1]{Joyce}}.
{\it Let $X$ be a set and let $F(X)$ be the free group generated by $X$. 
Denote by $FQ(X)$ the union of conjugation classes of elements of $X$:
\[FQ(X)= \bigcup_{x\in X} x^{F(X)}.\]
Consider $FQ(X)$ as a subquandle of ${\sf Conj}(F(X)).$ Then $FQ(X)$ is a free quandle generated by $X$.}

\section*{\bf Proof}

A subset of a group is called {\it independent} if it is a basis of a free subgroup.  

\begin{Lemma}\label{lemma_if_S_ind}
Let $S$ be a subset of $FQ(X)$ which is independent in $F(X).$
Then the subquandle of $FQ(X)$ generated by  $S$ is free.
\end{Lemma}
\begin{proof}
Consider the group homomorphism $F(S)\to F(X)$ induced by the embedding $S\hookrightarrow F(X).$ Since $S$ is independent, the homomorphism is injective. Then the restriction to the free quandles $FQ(S)\to FQ(X)$ is also injective. The subquandle generated by $S$ in $FQ(X)$ is equal to the image of the injective quandle morphism $FQ(S)\to FQ(X),$ and hence, it  is free. 
\end{proof}

Fix a set $X$ and denote $F:=F(X).$ We threat an element $w$ of $F$ as a reduced word and denote by $w_i$ its $i$th factor:
$$w=w_1w_2\dots w_n, \hspace{1cm} w_i\in X\cup X^{-1},\ \  w_i\ne w_{i-1}^{-1}.$$
The length of the word is denoted by $|w|:=n.$

Let $Q$ be a subquandle of $FQ(X)$. We are going to construct a subset $S(Q)$ in $Q$ and prove that it is a basis of $Q.$
 First, for an element $x\in X$ we consider the following subset:
\[ P_x:=\{ w\in F \mid x^w\in Q \wedge \  (w=1 \vee w_1\ne x^{\pm 1} )  \}. \]
It consists of all reduced words $w$ such that $x^w\in Q$ and whose first factor, if it exists, differs from $x$ and $x^{-1}.$  Then we define the set $T_x$ as follows
\[T_x = \left\{ w \in P_x \mid  \forall q\in Q \ \forall \varepsilon\in \{1,-1\} \ \ \  |wq^\varepsilon| > |w|  \right\}.\]
In some sense, $T_x$ consists of ``non shrinkable'' elements of $P_x.$ 
Finally, we consider the set 
\[ S(Q) :=  \bigcup_{x \in X}x^{T_x},\]
that consists of all elements of the form $x^w,$ where $w\in T_x.$ It is easy to see that $S(Q)\subseteq Q.$ Our aim is to prove that $Q$ is a free quandle generated by   $S(Q).$

\begin{Lemma}\label{lemma_SQ_gen}
The set $S(Q)$ generates the quandle $Q.$
\end{Lemma}
\begin{proof} Set $S:=S(Q)$ and denote by $\langle S \rangle$ the subquandle of $Q$ generated by $S.$ We need to prove that $Q=\langle S \rangle.$ Take an element $q\in Q.$ By construction, there is an element $x\in X$ and $w\in P_x$ such that $q=x^w.$ So we need to prove that $x^w\in Q \Rightarrow x^w\in \langle S\rangle.$ In order to prove this by induction, we reformulate this in the following  way: for any $n\geq 0,$ if $|w|= n, x\in X$ and $x^w\in Q,$ then $ x^w \in  \langle S\rangle.$ Prove this by induction on $n.$ 

Prove the base case. Assume that $n=0.$ Then $w=1.$ Hence, $1\in P_x$ and $1\in T_x.$ So $x\in x^{T_x}\subseteq S.$ Therefore $x\in \langle S \rangle.$  

Prove the induction step. Assume that $|w|=n, x\in X$ and $x^w\in Q.$ If $w\in T_x,$ then the statement is obvious, because $x^w\in S.$ Then we can assume that $w\notin T_x.$ In this case there exists $q\in Q$ such that $|wq^\varepsilon|\leq |w|$ for some $\varepsilon\in \{1,-1 \}.$ Note that  elements from $FQ(X)$ have odd lengths, and hence,  $|wq^\varepsilon|$ and $|w|$ have different parity. Then they can't be equal.  Therefore $|wq^\varepsilon|<n.$ Then we set $w':=wq^\varepsilon.$ Then $x^{w'}=x^w\triangleright^{\varepsilon} q,$ where $\triangleright=\triangleright$ and $\triangleright^{-1}=\triangleleft.$  We obtain $x^{w'}\in Q.$ By induction hypothesis, we have that $x^{w'}\in \langle S \rangle.$ Since $q\in Q,$ there exists $y\in X$ and $u\in P_y$ such that $q=y^u.$ Note that the first letter of $u$ is not equal to $y,$ and hence, there are no cancelations in the product $u^{-1}y^\varepsilon u.$  Since the length of $w'=wu^{-1}y^\varepsilon u$ is less than the length of $w,$ we obtain that $u^{-1}y^\varepsilon$ completely cancels in the product $wu^{-1}y^\varepsilon.$  Therefore, $|u|<|w|=n.$ By the induction hypothesis we obtain $q=y^u\in \langle S\rangle.$ Combining this with the fact that $x^{w'}\in \langle S \rangle$  and the equation $x^w=x^{w'} \triangleright^{\varepsilon} q,$ we obtain $x^w\in \langle S \rangle.$
\end{proof}

\begin{Lemma}\label{lemma_SQ_indep}
The set $S(Q)$ is independent in $F(X).$
\end{Lemma}
\begin{proof}
Following M. Hall \cite[\S 7.2]{Hall} we say that a subset $Y$ of $F(X)$ such that $Y\cap Y^{-1}=\emptyset$ {\it posses significant factors} if there is a collection of indexes $\{i(w)\}_{w\in Y\cup Y^{-1}}$ such that $0\leq i(w)\leq |w|,$   $i(w^{-1})=|w|+1-i(w)$
and in each product $wv$ for $w,v\in Y\cup Y^{-1}, w\ne v^{-1}$  the cancelation doesn't reach the factors $w_{i(w)}$ and $v_{i(v)}.$  If a subset possess significant factors, then it is independent  \cite[Th. 7.2.2.]{Hall}. So it suffices to show that $S(Q)$ possess significant factors. For each $x^w\in S(Q)$ we chose the central factor $x$ as a significant factor $i(x^w):=|w|+1.$ Then we only need to prove that the central factors $x^\varepsilon,y^\delta$  of the words $x^{\varepsilon w} $ and $y^{\delta v}$ do not cancel in the product $x^{\varepsilon w}y^{ \delta v}$ for $x,y\in X, w\in T_x,$
$ v\in T_y, \varepsilon,\delta\in \{1,-1\},$ 
and $x^{\varepsilon w}\ne y^{-\delta v}.$ 

Assume the contrary, that one of the factors $x^\varepsilon,y^\delta $ cancels in the product $w^{-1}x^\varepsilon wv^{-1} y^\delta v$. Then one of the following holds 
\begin{enumerate}
\item $v^{-1}$ can be presented as a product without cancelations  $v^{-1}=w^{-1}x^{-\varepsilon}u$ for some $u;$
\item  $w$ can be presented as a product without cancelations $w=uy^{-\delta}v$ for some $u.$
\end{enumerate} 
In the first case we have $vx^{-\varepsilon w}=u^{-1}w,$ and hence $|vx^{-\varepsilon w}|<|v|,$ which contradicts to the fact that $v\in T_y.$ In the second case we have $wy^{-\delta v}=uv,$ and hence $|wy^{-\delta v}|<|w|,$ which contradicts to the fact that $ w\in T_x.$
 \end{proof}

\begin{proof}[Proof of the theorem] Let $Q$ be a subquandle of $FQ(X).$ By Lemma \ref{lemma_SQ_gen} $Q$ is generated by the set $S(Q),$ which is independent by Lemma \ref{lemma_SQ_indep}. Then by Lemma \ref{lemma_if_S_ind}, we obtain that $Q$ is free.   
\end{proof}


\begin{thebibliography}{99}
 
\bibitem{Bardakov} V. G. Bardakov, M. Singh, M. Singh. Free quandles and knot quandles are residually free. Proc. Amer. Math. Soc. (2019)

\bibitem{Hall} M. Hall. The theory of groups. 

\bibitem{Joyce} D. Joyce. A classifying invariant of knots, the knot quandle. J. Pure Appl. Alg., 23, 37–65. (1982)





\end{thebibliography}
\end{document}